\documentclass{amsart}
\usepackage{amssymb}
\usepackage{mathrsfs}
\usepackage{amscd} 
\usepackage{hyperref}
\usepackage{graphicx}
\usepackage[curve]{xypic}
\usepackage[usenames,dvipsnames,x11names]{xcolor}
\usepackage{tikz}
\usetikzlibrary{arrows,snakes,shapes,shadows}
\usepackage{verbatim}

\newbox\mybox
\def\overtag#1#2#3{\setbox\mybox\hbox{$#1$}\hbox to
  0pt{\vbox to 0pt{\vglue-#3\vglue-\ht\mybox\hbox to \wd\mybox
      {\hss$\ss#2$\hss}\vss}\hss}\box\mybox}
\def\undertag#1#2#3{\setbox\mybox\hbox{$#1$}\hbox to 0pt{\vbox to
    0pt{\vglue#3\vglue\ht\mybox\hbox to \wd\mybox
      {\hss$\ss#2$\hss}\vss}\hss}\box\mybox}
\def\lefttag#1#2#3{\hbox to 0pt{\vbox to 0pt{\vglue -6pt\hbox to
      0pt{\hss$\ss#2$\hskip#3}\vss}}#1}
\def\righttag#1#2#3{\hbox to 0pt{\vbox to 0pt{\vglue -6pt\hbox to
      0pt{\hskip#3$\ss#2$\hss}\vss}}#1}
\let\ss\scriptstyle
\def\splicediag#1#2{\xymatrix@R=#1pt@C=#2pt@M=0pt@W=0pt@H=0pt}
\def\Dot{\lower.2pc\hbox to 2pt{\hss$\bullet$\hss}}
\def\Circ{\lower.2pc\hbox to 2pt{\hss$\circ$\hss}}
\def\Vdots{\raise5pt\hbox{$\vdots$}}
\newcommand\lineto{\ar@{-}}
\newcommand\dashto{\ar@{--}}
\newcommand\dotto{\ar@{.}}

\renewcommand{\setminus}{\smallsetminus}
\newcommand\Q{{\mathbb Q}}

\newcommand\C{{\mathbb C}}

\newtheorem{theorem}{Theorem}[section]
\newtheorem{proposition}[theorem]{Proposition}
\newtheorem*{theorem*}{Theorem}
\newtheorem{corollary}[theorem]{Corollary}
\newtheorem{lemma}[theorem]{Lemma}

\theoremstyle{definition}

\newtheorem*{remark*}{Remark}

\renewcommand{\int}{\operatorname{int}}
\newcommand{\lcm}{\operatorname{lcm}}

\begin{document}
\title[ Lipschitz geometry of complex curves]{ Lipschitz geometry of complex curves}
\author{Walter D Neumann}\address{Department of Mathematics, Barnard
  College, Columbia University, 2009 Broadway MC4424, New York, NY
  10027, USA} \email{neumann@math.columbia.edu} \author{Anne Pichon}
\address{Aix-Marseille Univ, IML\\ FRE 3529 CNRS,
   Campus de Luminy - Case 907\\ 13288
  Marseille Cedex 9, France} \email{anne.pichon@univ-amu.fr}
\subjclass{14B05, 32S25, 32S05, 57M99} \keywords{bilipschitz,
  Lipschitz geometry, complex curve singularity, embedded topological type}
\begin{abstract}  We describe the Lipschitz geometry of complex
  curves. For the most part this is well known material, but we give a
  stronger version even of known results. In particular, we give a
  quick proof, without any analytic restrictions, that the outer
  Lipschitz geometry of a germ of a complex plane curve determines and
  is determined by its embedded topology. This was first proved by
  Pham and Teissier, but in an analytic category.
 \end{abstract}
 \maketitle
 
\section{Introduction} 

The germ of a complex set $(X,0)\subset (\C^N,0)$ has two metrics
induced from the standard hermitian metric on $\C^N$: the \emph{outer
  metric} given by distance in $\C^N$ and the \emph{inner metric}
given by arc-length of curves on $X$. Both are well defined up to
bilipschitz equivalence, {\it i.e.}, they only depend on the analytic
type of the germ $(X,0)$ and not on the embedding $(X,0)\subset
(\C^N,0)$. Studies of what information can be extracted from this
metric structure have generally worked under analytic restrictions,
e.g., that equivalences be restricted to be analytic or semi-algebraic
or similar. In this note we prove the metric classification of germs
of complex plane curves, but without any analytic restrictions
(equivalence of item \eqref{it1} of the following theorem with the
other items):

\begin{theorem} \label{th:main} Let $(C_1,0)\subset (\C^2,0)$ and
  $(C_2,0)\subset (\C^2,0)$ be two germs of complex curves. The
  following are equivalent:
   \begin{enumerate}
  \item\label{it1} $(C_1,0)$ and $(C_2,0)$ have same Lipschitz
    geometry, {\it i.e.}, there is a homeo\-morphism of germs $\phi\colon
    (C_1,0)\to (C_2,0)$ which is bilipschitz for the outer metric;
  \item\label{it2} there is a homeomorphism of germs
    $\phi\colon (C_1,0)\to (C_2,0)$, holomorphic
    except at $0$,  which is bilipschitz for the outer
    metric;
  \item\label{it3} $(C_1,0)$ and $(C_2,0)$ have the same embedded
    topology, {\it i.e.}, there is a homeo\-morphism of germs $h \colon
    (\C^2,0) \to (\C^2,0)$ such that $h(C_1)=C_2$;
  \item\label{it4} there is a bilipschitz homeomorphism
    of germs $h\colon (\C^2,0) \to (\C^2,0)$ with $h(C_1)=C_2$.
  \end{enumerate}
\end{theorem}
The equivalences of \eqref{it1} and \eqref{it4} with \eqref{it3} are
our new contributions. The equivalence of \eqref{it2} and \eqref{it3}
was first proved by Pham and Teissier \cite{pham-teissier}. By
Teissier \cite[Remarque, p.354]{T3} (see also Fernandes
\cite{fernandes}) it then also follows that the outer bilipschitz
geometry of any curve germ $(X,0)\subset (\C^N,0)$ determines the
embedded topology of its general plane projection (Corollary
\ref{cor:main}).

For completeness we give quick proofs of all the equivalences. We
start with the result for inner geometry, which will be
used in examining outer geometry.

\subsection*{Acknowledgments} We thank Bernard Teissier for helpful
comments on the first version of this paper. This research was supported
by NSF grant DMS12066760 and by  the  ANR-12-JS01-0002-01 SUSI.
\section{Inner geometry}\label{sec:inner}

An algebraic germ $(X,0)\subset (\C^N,0)$ is homeomorphic to the cone
on its link $X\cap S_\epsilon$, where $S_\epsilon$ is the sphere of
radius $\epsilon$ about the origin with $\epsilon$ sufficiently small.
If it is endowed with a metric, it is \emph{metrically conical} if it
is bilipschitz equivalent to the metric cone on its link. This
basically means that the metric tells one no more than the topology
(and is therefore uninteresting).

\begin{proposition} \label{prop:inner} Any space curve germ $(C,0)
  \subset (\C^N,0)$ is metrically conical for the inner geometry.
\end{proposition}

\begin{proof}
  Take a linear projection $p \colon \C^N \to \C$ which is generic for
  the curve $(C,0)$ ({\it i.e.}, its kernel contains no
  tangent line of $C$ at $0$). Its restriction $p |_C \colon C \to \C$
  is a branched cover
  branched at $0$ which is bilipschitz for the inner geometry.  So $C$ is 
  metrically conical since it is
  strictly conical  with the metric lifted from $\C$.
\end{proof}

\section{Outer geometry determines embedded topological
  type}\label{sec:outer implies embedded}
 
In this section, we prove \eqref{it1} $\Rightarrow$ \eqref{it3} of Theorem
\ref{th:main}, {\it i.e.}, that the embedded topological type of a
plane curve germ $(C,0) \subset (\C^2,0)$ is determined by the outer
Lipschitz geometry of $(C,0)$.
  
We first prove this using the analytic structure and the outer metric
on $(C,0)$. The proof is close to Fernandes' approach in
\cite{fernandes}.  We then modify the proof to
make it purely topological and to allow a bilipschitz change of the
metric.

The tangent space to $C$ at $0$ is a union of lines $L^{(j)}$,
$j=1,\dots,m$, and by choosing our coordinates we can assume they are
all transverse to the $y$-axis.

There is $\epsilon_0 >0$ such that for any $\epsilon\le \epsilon_0$ the
curve $C$  meets transversely the set
$$T_\epsilon:=\{(x,y)\in \C^2: |x|= \epsilon\}\,.$$

Let $\mu$ be the multiplicity of $C$. The lines $x=t$ for $t\in
(0,\epsilon_0]$ intersect $C$ in $\mu$ points
$p_1(t),\dots,p_{\mu}(t)$ which depend continuously on $t$.  Denote by $[\mu]$ the set
$\{1,2,\dots,\mu\}$. For each
$  j, k \in [\mu]$ with $j < k$, the distance $d(p_j(t),p_k(t))$ has the form
$O(t^{q(j,k)})$, where $q(j,k) = q(k,j)$ is either a characteristic Puiseux
exponent for a branch of the plane curve $C$ or a coincidence exponent
between two branches of $C$ in the sense of e.g., \cite{LMW}. We call
such exponents \emph{essential}.  For $j\in [\mu]$ define $q(j,j)=\infty$.

\begin{lemma} \label{le:curve geometry}The map $q\colon
  [\mu]\times[\mu]\to \Q\cup\{\infty\}$, $(j,k)\mapsto q(j,k)$,
  determines the embedded topology of $C$.
\end{lemma}

\begin{proof} There are many combinatorial objects that encode the
  embedded topology of $C$, for example the  Eisenbud-Neumann splice diagram
  \cite{EN} of the curve or the Eggers tree \cite{Eggers} (both are
  described, with the relationship between them, in
  C.T.C. Wall's book \cite{CTCWall}). The combinatorial
    carrousel, introduced in \cite{NP}, is closely related. All
  three are rooted trees with edges or vertices decorated
  with numeric labels.

  To prove the lemma we will construct the combinatorial carrousel
  from $q$. We also describe how one derives the splice diagram from
  it.

  The $q(j,k)$ have the property that $q(j,l) = min (q(j,k),q(k,l))$
  for any triple $j,k,l$ with $j\ne l$. So for any $q\in \Q\cup\{\infty\}, q>0$, the
  relation on the set $[\mu]$ given by $j\sim_q k\Leftrightarrow
  q(j,k)\ge q$ is an equivalence relation.

  Name the elements of the set $q([\mu]\times[\mu])\cup\{1\}$ in
  decreasing order of size: $\infty=q_0>q_1>q_2>\dots>q_s=1$.
  For each $i=0,\dots,s$ let $G_{i,1},\dots,G_{i,\mu_i}$ be the
  equivalence classes for the relation $\sim_{q_i}$. So $\mu_0=\mu$
  and the sets $G_{0,j}$ are singletons while $\mu_s=1$ and
  $G_{s,1}=[\mu]$. We form a tree with these equivalence classes $G_{i,j}$ as
  vertices, and edges given by inclusion relations:  the
  singleton sets $G_{0,j}$ are the leaves and there is an edge between
  $G_{i,j}$ and $G_{i+1,k}$ if $G_{i,j}\subseteq G_{i+1,k}$. The vertex
  $G_{s,1}$ is the root of this tree. We weight each vertex   with its corresponding $q_i$.

  The combinatorial carrousel is the tree obtained from this
 tree by suppressing valence $2$ vertices: we remove each such
  vertex and amalgamate its two adjacent edges into one edge. We will
  describe how one gets from this to the splice diagram, but we
  first give an illustrative example.

We will use the plane curve $C$ with two branches given by
$$y=x^{3/2}+x^{13/6},\quad y=x^{7/3}\,.
$$ 
Here are pictures of sections of $C$ with complex lines
 $x=0.1$, $0.05$, $0.025$ and $0$. { The central three-points set corresponds to the  branch $y=x^{7/3}$ while the two lateral  three-points sets correspond to the other branch. }

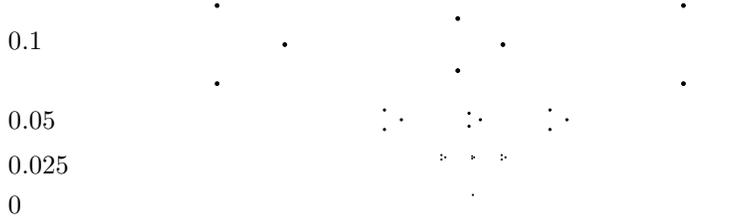
\begin{figure}[h]
  \centering
\vbox to 0 pt{\vglue12pt\hbox to 0 pt{$0.1$\hss}
\vglue18pt\hbox to 0 pt{$0.05$\hss}
\vglue5pt\hbox to 0 pt{$0.025$\hss}
\vglue3pt\hbox to 0 pt{$0$\hss}\vss}
\begin{tikzpicture}
\draw[fill=black] (3.1,0)+(0:.6)circle(.7pt);
\draw[fill=black] (3.1,0)+(120:.6)circle(.7pt);
\draw[fill=black] (3.1,0)+(240:.6)circle(.7pt);
\draw[fill=black] (-3.1,0)+(0:.6)circle(.7pt);
\draw[fill=black] (-3.1,0)+(120:.6)circle(.7pt);
\draw[fill=black] (-3.1,0)+(240:.6)circle(.7pt);
\draw[fill=black] (0,0)+(0:0.4)circle(.7pt);
\draw[fill=black] (0,0)+(120:0.4)circle(.7pt);
\draw[fill=black] (0,0)+(240:0.4)circle(.7pt);
\draw[fill=black] (1.1,-1)+(0:.15)circle(.4pt);
\draw[fill=black] (1.1,-1)+(120:.15)circle(.4pt);
\draw[fill=black] (1.1,-1)+(240:.15)circle(.4pt);
\draw[fill=black] (-1.1,-1)+(0:.15)circle(.4pt);
\draw[fill=black] (-1.1,-1)+(120:.15)circle(.4pt);
\draw[fill=black] (-1.1,-1)+(240:.15)circle(.4pt);
\draw[fill=black] (0,-1)+(0:0.1)circle(.4pt);
\draw[fill=black] (0,-1)+(120:0.1)circle(.4pt);
\draw[fill=black] (0,-1)+(240:0.1)circle(.4pt);
\draw[fill=black] (.4,-1.5)+(0:.034)circle(.2pt);
\draw[fill=black] (.4,-1.5)+(120:.034)circle(.2pt);
\draw[fill=black] (.4,-1.5)+(240:.034)circle(.2pt);
\draw[fill=black] (-.4,-1.5)+(0:.034)circle(.2pt);
\draw[fill=black] (-.4,-1.5)+(120:.034)circle(.2pt);
\draw[fill=black] (-.4,-1.5)+(240:.034)circle(.2pt);
\draw[fill=black] (0,-1.5)+(0:0.018)circle(.15pt);
\draw[fill=black] (0,-1.5)+(120:0.018)circle(.15pt);
\draw[fill=black] (0,-1.5)+(240:0.018)circle(.15pt);
\draw[fill=black] (0,-2)+(0:0)circle(.2pt);
\end{tikzpicture}
\end{figure}
The combinatorial carrousel for this example is the tree on the left
in the picture below and the procedure we will describe for getting from
it to the splice diagram is then illustrated
in the middle and right trees.  We will follow the computer science
convention of drawing the tree with its root vertex at the top,
descending to its leaves at the bottom.
$$
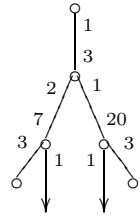
\splicediag{10}{3}{&&&&\hbox to 0 pt{\hss \Small{Combinatorial
      carrousel}\hss}\\
&&&&\righttag{\Circ}{1}{4pt}\lineto[dd]\\\\
&&&&\righttag{\Circ}{\frac32}{4pt}\lineto[ddddlll]\lineto[dd]\lineto[ddrrr]\\\\
&&&&\righttag{\Circ}{\frac{13}6}{3pt}\ar@{-}[ddddl]\ar@{-}[dddd]\ar@{-}[ddddr]&&&
\righttag{\Circ}{\frac{13}6}{4pt}\ar@{-}[ddddl]\ar@{-}[dddd]\ar@{-}[ddddr]\\\\
&\lefttag{\Circ}{\frac73}{3pt}\ar@{-}[ddl]\ar@{-}[dd]\ar@{-}[ddr]\\\\
\Circ&\Circ&\Circ& \Circ&\Circ&\Circ& \Circ&\Circ&\Circ& }\quad\quad\quad\quad\quad\quad
\splicediag{10}{3}{\\&&&&\righttag{\Circ}{1}{4pt}\lineto[dd]\\\\
&&&&\righttag{\Circ}{\frac32}{4pt}\lineto[ddddlll]_(.25){2}\lineto[ddrrr]\\\\
&&&&&&&\righttag{\Circ}{\frac{13}6}{4pt}\ar@{-}[dddd]\\\\
&\lefttag{\Circ}{\frac{14}6}{2pt}\ar@{-}[dd]&&&&&&&&&\\\\
&\Circ&&&&&&\Circ }
\quad\splicediag{9}{3}{\hbox to 0 pt{\hss\Small{Eggers tree}\hss}\\\\\\\\=}\quad
\splicediag{10}{3}{\\\\
&&&&\righttag{\Circ}{1}{4pt}\lineto[dd]\\\\
&&&&\righttag{\Circ}{\frac32}{4pt}\lineto[ddlll]_(.4){2}\lineto[ddrrr]\\\\
&\lefttag{\Circ}{\frac{14}6}{2pt}\ar@{-}[dd]&&&&&&
\righttag{\Circ}{\frac{13}6}{4pt}\ar@{-}[dd]\\\\
&\Circ&&&&&&\Circ }\quad\quad\quad\quad
\splicediag{10}{3}{
&&&&&&\hbox to 0 pt{\hss\Small{Splice diagram}\hss}\\\\
&&&&&&&\Circ\lineto[dd]^(.35){1}^(.8){3}\\\\
&&&&&&&\Circ\lineto[ddlll]_(.35){2}_(.8)7\lineto[ddrrr]^(.3)1^(.8){20}\\\\
&&&&\Circ\ar@{->}[dd]^(.3)1\ar@{-}[dlll]_(.35)3&&&&&&
\Circ\ar@{->}[dd]_(.3)1\ar@{-}[drrr]^(.4)3\\
&\Circ&&&&&&&&&&&&\Circ\\
&&&&&&&&&&&&\\
&&&&&&&}\quad\quad
$$

At any non-leaf vertex $v$ of the combinatorial carrousel
  we have a weight $q_v$, $1\le q_v\le q_1$, which is one of the
  $q_i$'s. We write it as $m_v/n_v$, where $n_v$ is the $\lcm$ of the
  denominators of the $q$-weights at the vertices on the path from $v$
  up to the root vertex. If $v'$ is the adjacent vertex above $v$
  along this path, we put $r_v=n_v/n_{v'}$ and $s_v=n_v(q_v-q_{v'})$.
  At each vertex $v$ the subtrees cut off below $v$ consist of groups
  of $r_v$ isomorphic trees, with possibly one additional tree.  We
  label the top of the edge connecting to this additional tree at $v$,
  if it exists, with the number $r_v$, and then delete all but one from
  each group of $r_v$ isomorphic trees below $v$. We do this for each
  non-leaf vertex of the combinatorial carrousel. The resulting tree,
  with the $q_v$ labels at vertices and the extra label on a downward
  edge at some vertices is easily recognized as a mild modification of
  the Eggers tree.

  We construct the splice diagram starting from this tree. We first
  replace every leaf by an arrowhead. Then at each vertex $v$ which
  did not have a downward edge with an $r_v$ label we add such an edge
  (ending in a new leaf which is not an arrowhead).  Each still
  unlabeled top end of an edge is then given the label $1$. Finally,
  starting from the top of the tree we move down the tree adding a
  label to the bottom end of each edge ending in a vertex $v$ which is
  not a leaf as follows. If $v$ is directly below the root the label
  is $m'_v:=m_v$. For a vertex $v$ directly below a vertex $v'$ other
  than the root the label is $m'_v:=s_v+r_vr_{v'}m'_{v'}$ if 
  $r_{v'}$ does not label the edge $v'v$ and
  $m'_v:=(s_v+r_vm'_{v'})/r_{v'}$ if it does (see \cite[Prop.\ 1A.1]{EN}).
\end{proof}

As already noted, this discovery of the embedded topology involved the
complex structure and outer metric. We must show we can discover
it without use of the complex structure, even after applying a
bilipschitz change to the outer metric.

Recall that the tangent space of $C$ is a union of lines $L^{(j)}$. We
denote by $C^{(j)}$ the part of $C$ tangent to the line $L^{(j)}$.  It
suffices to discover the topology of each $C^{(j)}$ independently,
since the $C^{(j)}$'s are distinguished by the fact that the distance
between any two of them outside a ball of radius $\epsilon$ around $0$
is $O(\epsilon)$, even after bilipschitz change to
the metric. We therefore assume from now on that the tangent to
$C$ is a single complex line.

The points $p_1(t),\dots,p_\mu(t)$ we used to find the numbers
$q(j,k)$ were obtained by intersecting $C$ with the line $x=t$. The arc
$p_1(t)$, $t\in [0,\epsilon_0]$ satisfies $d(0,p_1(t))=O(t)$. Moreover,
the other points $p_2(t),\dots,p_\mu(t)$ are in the transverse disk of 
radius $rt$ centered at $p_1(t)$ in the plane $x=t$. Here $r$ can be as
small as we like, so long as $\epsilon_0$ is then chosen sufficiently
small.

Instead of a transverse disk of radius $rt$, we can use a ball
$B(p_1(t),rt)$ of radius $rt$ centered at $p_1(t)$. This
$B(p_1(t),rt)$ intersects $C$ in $\mu$ disks $D_1(t),\dots,D_\mu(t)$,
and we have $d(D_j(t),D_k(t))=O(t^{q(j,k)})$, so we still recover the
numbers $q(j,k)$.  In fact, the ball in the outer metric on $C$ of
radius $rt$ around $p_1(t)$ is $B_C(p_1(t),rt):=C\cap B(p_1(t),rt)$,
which consists of these $\mu$ disks $D_1(t),\dots,D_\mu(t)$.

We now replace the arc $p_1(t)$ by any continuous arc $p'_1(t)$ on
$C$ with the property that $d(0,p'_1(t))=O(t)$, and if $r$ is
sufficiently small it is still true that $B_C(p'_1(t),rt)$ consists of
$\mu$ disks $D'_1(t),\dots,D'_\mu(t)$ with
$d\bigl(D'_j(t),D'_k(t)\bigr)=O(t^{q(j,k)})$.  So at this point, we have gotten
rid of the dependence on analytic structure in discovering the
topology, but not yet dependence on the outer geometry.

A $K$-bilipschitz change to the metric may make the components of
$B_C(p'_1(t),rt)$ disintegrate into many pieces, so we can no longer
simply use distance between pieces. To resolve this, we consider both
$B_C'(p'_1(t),rt)$ and $B_C'(p'_1(t),\frac r{K^4}t)$ where $B'$ means
we are using the modified metric. Then only $\mu$ components of
$B_C'(p_1(t),rt)$ will intersect $B_C'(p_1(t),\frac r{K^4}t)$.  Naming
these components $D'_1(t),\dots,D'_\mu(t)$ again, we still have
$d(D'_j(t),D'_k(t))=O(t^{q(j,k)})$ so the $q(j,k)$ are determined as
before.\qed

\section{Embedded topological type determines outer geometry}

In this section, we prove \eqref{it3} $\Rightarrow$ \eqref{it2} of
Theorem \ref{th:main}. The implication \eqref{it2} $\Rightarrow$
\eqref{it1} is trivial, so we then have the equivalence of the first
three items of Theorem \ref{th:main}.

We will use the  following lemma: 
\begin{lemma}\label{le:curve}
  Let $(C,0) \subset (\C^N,0)$ be a germ of complex plane curve and
  let $p \colon \C^N \to \C$ be a linear projection whose kernel does
  not contain any tangent line to $C$.  Then there exists a
  neighborhood $U$ of $0$ in $C$ and a constant $M>1$ such that for
  each $u,u' \in U\setminus\{0\}$, there is an arc $\tilde \alpha$ in
  $C$ joining $u$ to a point $u''$ with $p(u'')=p(u')$ and
$$ d(u,u') \leq L(\tilde \alpha) + d(u'',u') \leq M d(u,u')$$ 
where $L( \tilde \alpha)$ denotes the length of $\tilde \alpha$.
 \end{lemma}

\begin{proof}
  There exists a neighbourhood $U$ of $0$ in $C$ such that the
  restriction $p|_{C}$ is a bilipschitz local homeomorphism for the
  inner metric on $U \setminus \{0\}$ (see proof of Proposition
  \ref{prop:inner}).  Choose any $\delta>1$. If $0$ is not in
  the segment $[p(u),p(u')]$, we set $\alpha = [p(u),p(u')]$. If $0
  \in [p(u),p(u')]$, we modify this segment to a curve $\alpha$
  avoiding $0$ which has length at most $\delta$ times the length of
  $[p(u),p(u')]$.  Consider the lifting $\tilde \alpha$ of $\alpha$ by
  $p|_{C}$ with origin $u$ and let $u''$ be its extremity.  We
  obviously have:
$$d(u,u')\leq L( \tilde\alpha) + d(u',u'')\,.$$
 
On the other hand, $L( \tilde\alpha) \leq K_0 L(\alpha)\leq \delta K_0
d(p(u),p(u'))$, where $K_0$ is a bound for the local inner
bilipschitz constant of $p$ on $U \setminus \{0\}$.  As $d(p(u),p(u'))
\leq d(u,u')$, we then obtain: $L( \tilde\alpha) \leq \delta K_0 
d(u,u')$

If we join the segment $[u,u']$ to
$ \tilde\alpha$ at $u$ we get a curve from $u'$ to
$u''$, so   $d(u',u'')  \leq (1+\delta K_0) d(u,u')$.  We then obtain: 
$$ L( \tilde\alpha) + d(u',u'')  \leq (1+2 \delta K_0 ) d(u,u'),$$
and $M=1+2 \delta K_0 $ is the desired constant. 
\end{proof}

\begin{proof}[Proof of \eqref{it3} $\Rightarrow$ \eqref{it2} of
  Theorem \ref{th:main}]

 Let $(C_1,0) \subset (\C^2,0)$ 
be an 
irreducible plane curve which is not tangent to the $y$-axis. Then
there exists a minimal integer $n >0$ such that $(C_1,0)$ has 
Puiseux parametrization
$$\gamma_1(w)=(w^n,\sum_{i \geq n} a_i w^i)\,.$$
Denote $A: = \{i : a_i \neq 0 \}$. 
Recall that the embedded topology of $C_1$ is
determined by $n$ and the essential integer exponents in the sum $\sum_{i \geq
  n} a_i w^i$, where an $i\in A\setminus\{n\}$ is an \emph{essential
  integer exponent} if
and only if
$ \gcd\{j\in \{n\}\cup A: j\le i\}<\gcd\{j\in \{n\}\cup
A:j<i\}$ (equivalently $\frac in$ is a characteristic exponent). Denote by $A_e$ 
the subset of $A$ 
consisting of the essential integer exponents.

Now let $(C_2,0)\subset (\C^2,0)$, given by 
$$\gamma_2(w)=(w^n,\sum_{i \geq n} b_i w^i)\,,$$
 be a second plane curve with the
same embedded topology as $C_1$, so that the set of essential integer
exponents $B_e\subset B:= \{i :b_i \neq 0 \}$ is equal to $A_e$. 

We will prove that the homeomorphism $\Phi \colon C_1 \to C_2$ defined
by $\Phi(\gamma_1(w))=\gamma_2(w)$ is bilipschitz on small
neighborhoods of the origin.

We first prove that there exists $K>0$ and a neighborhood $U$ of
$0$ in $\C$ such that for each pair $(w,w')$ with $w\in U$, $w\ne w'$ and
$w^n=(w')^n$, we have
$$ d\bigl(\gamma_1(w),\gamma_1(w')\bigr) 
\leq K d\bigl(\gamma_2(w),\gamma_2(w')\bigr)$$

For $(w,w')$ as above, consider the two real arcs $s \in [0,1] \mapsto
\gamma_1(sw)$ and $s \mapsto \gamma_1(sw')$ and their images by
$\Phi$. Then we have
$$d\bigl(\gamma_1(ws), \gamma_1(w's)\bigr) = s^n   
\bigg|    \sum_{i> n}a_i s^{i-n}\bigl(w^i - (w')^i\bigr)  \bigg|$$
and
$$d\Bigl(\Phi\bigl(\gamma_1(ws)\bigr),
\Phi\bigl(\gamma_1(w's)\bigr)\Bigr) = s^n   
\bigg| \sum_{i > n}b_j s^{i-n}\bigl(w^i - (w')^i\bigr)
\bigg| $$
 
Let $i_0$ be the minimal element of $\{i \in A ; w^i \neq
(w')^i\}$. Then $i_0$ is an essential integer exponent, so $a_{i_0}$ and
$b_{i_0}$ are non-zero.  Moreover, as $s$ tends to $0$ we have  $d\bigl(\gamma_1(ws),
\gamma_1(w's)\bigr) \sim s^{i_0} |w^{i_0}-(w')^{i_0}| |a_{i_0}|$ and $
d\bigl(\Phi\bigl(\gamma_1(ws)\bigr),
\Phi\bigl(\gamma_1(w's)\bigr)\bigr) \sim s^{i_0}
|w^{i_0}-(w')^{i_0}||b_{i_0}|$ and hence the ratio
$$ d\Bigl(\gamma_1(ws), \gamma_1(w's)\Bigr)\,\Big/\, 
d\Bigl(\Phi\bigl(\gamma_1(ws)\bigr),
\Phi\bigl(\gamma_1(w's)\bigr)\Bigr) 
\hbox to 0cm{\hglue1.5cm$(\ast)$\hss}
$$ 
tends to the non zero constant $c_{i_0}= \frac{|a_{i_0}|}{|b_{i_0}|}$.

Notice that the integer $i_0$ depends on the pair of points
$(w,w')$. But $i_0$ is  either $n$ or an  essential integer exponent
for $\gamma_1$. Therefore there are a finite number of values for $i_0$
and $c_{i_0}$.  Moreover, the set of pairs $(w,w')$ such that
$w^n=(w')^n$ consists of a disjoint union of $n$ lines. So there
exists $s_0 >0$ such that for each such $(w,w')$ with $|w| = 1$ and
each $s \leq s_0$, the quotient $(\ast)$ belongs to $[1/K,K]$ where
$K>0$. Then $U = \{w: |w| \leq s_0\}$ is the desired neighbourhood of
$0$.

We now prove that $\Phi$ is bilipschitz on $\gamma_1(U)$. Consider the
projection $p\colon \C^2 \to \C$ given by $p(x,y)=x$. Let $w$ and $w'$
be any two complex numbers in $U$. Let $\alpha $ be the segment in
$\C$ joining $w^n$ to $(w')^n$ and let $\tilde{\alpha}_1$ (resp.\
$\tilde{\alpha}_1$) be the lifting of $\alpha$ by the restriction
$p|_{C_1}$ (resp.\ $p|_{C_2}$) with origin $ \gamma_1(w)$ (resp.\
$\gamma_2(w)$). Consider the unique $w'' \in \C$ such that
$\gamma_1(w'')$ is the extremity of $\tilde{\alpha}_1$.  Notice that
$\gamma_2(w'')$ is the extremity of $\tilde{\alpha}_2$. We have
$$
d\bigl(\gamma_1(w),\gamma_1(w')\bigr) \leq L(\tilde \alpha_1) +
d\bigl(\gamma_1(w''),\gamma_1(w')\bigr).$$

According to Section \ref{sec:inner}, $p|_{C_1}$ (resp.\ $p|_{C_2}$)
is an inner bilipschitz homeomorphism with bilipschitz constant say
$K_1$ (resp.\ $K_2$). We then have $ L(\tilde \alpha_1) \leq K_1 K_2
L(\tilde \alpha_2)$. Therefore setting $C=\max(K_1K_2,K)$, we obtain:
$$d\bigl(\gamma_1(w),\gamma_1(w')\bigr) \leq C\Bigl( L(\tilde \alpha_2) +  
d\bigl(\gamma_2(w''),\gamma_2(w')\bigr)\Bigr) \hbox to
0pt{\hglue1.3cm$(\ast \ast)$\hss}$$

Applying Lemma \ref{le:curve} to the restriction $p|_{C_2}$ with
$u=\gamma_2(w)$ and $u'=\gamma_2(w')$, we then obtain:
$$
d\bigl(\gamma_1(w),\gamma_1(w')\bigr) \leq C M
d\bigl(\gamma_2(w),\gamma_2(w')\bigr)$$

This proves $\Phi$ is Lipschitz. It is then bilipschitz by symmetry of
the roles.

In the general case where $C_1$ and $C_2$ are not necessarily
irreducible, the same arguments work taking into account a Puiseux
parametrization for each branch and the fact that the sets of
characteristic exponents and coincidence exponents between branches
coincide.
\end{proof}

\section{Outer geometry of space curves} \label{sec:non plane}
 
Before proving the final equivalence of Theorem \ref{th:main} we give
a quick proof, based on the preceding proof, of the following
result of Teissier \cite[pp. 352--354]{T3}.

\begin{theorem} \label{th:proj}For a complex curve germ $(C,0)
  \subset (\C^N,0)$ the restriction to $C$ of a generic linear 
  projection $\ell \colon \C^N \to \C^2$ is bilipschitz for the
  outer geometry.
\end{theorem}
Our notion of \emph{generic linear projection} to $\C^2$, defined in
the proof below, is equivalent to Teissier's, which says that the kernel of the
projection should contain no limit of secant lines to the curve.
\begin{proof}[Proof of Theorem \ref{th:proj}] We have to prove that
  the restriction $\ell |_C \colon C \to \ell(C) $ is bilipschitz for
  the outer metric. We choose coordinates $(x,y)$ in $\C^2$ so
  $\ell(C)$ is transverse to the $y$-axis at $0$  and coordinates
  $(z_1,\dots,z_n)$ in $\C^n$ with $z_1=x\circ\ell$. So $\ell$
has the form $(z_1,\dots,z_N)\mapsto (z_1, \sum_1^Nb_jz_j)$ and any 
component of
  $C$ has a Puiseux expansion of the form ($n$ is the multiplicity of
  the component):  
$$\gamma(w)=(w^n,\sum_{i\ge n}a_{2i}w^i,\dots,\sum_{i\ge
  n}a_{Ni}w^i)\,.$$ We first assume $(C,0)$ is irreducible. We again
denote $A:=\{i:\exists j, a_{ji}\ne 0\}$ and call an exponent $i\in A\setminus\{n\}$ an
\emph{essential integer exponent} if and only if
$$ \gcd\{j\in \{n\}\cup A: j\le i\}<\gcd\{j\in \{n\}\cup
A:j<i\}.$$
Define $a_{1n}=1$ and $a_{1i}=0$ for $i>n$. We say $\ell$ is 
\emph{generic} if  $\sum_{j=1}^Nb_ja_{ji}\ne 0$ for each
essential integer exponent $i$. We now assume $\ell$ is generic.

As in the proof of the second part of Theorem \ref{th:main} there
then exists $K>0$ and a neighborhood $U$ of $0$ in $\C$ such that for each
pair $(w,w')$ with $w\in U$ and $w^n=(w')^n$ we have 
$$
\frac1K d\bigl(\ell\gamma(w),\ell\gamma(w')\bigr)\le
d\bigl(\gamma(w),\gamma(w')\bigr)\le
Kd\bigl(\ell\gamma(w),\ell\gamma(w')\bigr)\,.$$ Lemma \ref{le:curve}
then completes the proof, as before.

The proof when $C$ is reducible is essentially the same, but the
genericity condition must take both characteristic and coincidence
exponents into consideration. Namely, $\ell$ should be generic
as above for each individual branch of $C$; and for any two branches, given by
(with $n$ now the $\lcm$ of their multiplicities)
$$\gamma(w)=(w^n,\sum_{i\ge n}a_{2i}w^i,\dots,\sum_{i\ge
  n}a_{Ni}w^i),\quad\gamma'(w)=(w^n,\sum_{i\ge
  n}a'_{2i}w^i,\dots,\sum_{i\ge n}a'_{Ni}w^i),$$ we require 
$\sum_{j=1}^Nb_j(a_{ji}-\lambda^ia'_{ji})\ne 0$ for each $n$-th root of
unity $\lambda$, where $i$ is the smallest exponent for which some
$a_{ji}-a'_{ji}$ is non-zero.
\end{proof}

\begin{corollary}\label{cor:main} Let $(C_1,0)\subset (\C^{N_1},0)$
  and $(C_{2},0)\subset (\C^{N_2},0)$ be two germs of complex
  curves. The following are equivalent:
\begin{enumerate}
\item $(C_1,0)$ and $(C_2,0)$ have same Lipschitz geometry {\it i.e.},
  there is a homeomorphism of germs $\phi\colon (C_1,0)\to (C_2,0)$
  which is bilipschitz for the outer metric;
\item there is a homeomorphism of germs $\phi\colon (C_1,0)\to
  (C_2,0)$, holomorphic except at $0$, which is bilipschitz for the
  outer metric;
\item the generic plane projections of $(C_1,0)$ and $(C_2,0)$ have
  the same embedded topology. \qed
\end{enumerate}
\end{corollary}

\section{Ambient geometry of plane curves}\label{sec:ambient}

To complete the proof of Theorem \ref{th:main} we must show the
implication \eqref{it3} $\Rightarrow$ \eqref{it4} of that theorem,
since \eqref{it4} $\Rightarrow$ \eqref{it3} is trivial.
We will use a \emph{carrousel decomposition} of $(\C^2,0)$ with respect
to a plane curve, so we first describe this (it is essentially the one
described in \cite{BNP}).

The tangent space to $C$ at $0$ is a union $\bigcup_{j=1}^mL^{(j)}$ of
lines. For each $j$ we denote the union of components of $C$ which are
tangent to $L^{(j)}$ by $C^{(j)}$.  We can assume our coordinates
$(x,y)$ in $\C^2$ are chosen so that no $L^{(j)}$ is tangent to an
axis. Then $L^{(j)}$ is given by an equation $y=a_1^{(j)}x$ with
$a_1^{(j)}\ne 0$.

We choose $\epsilon_0>0$ sufficiently small that the set $\{(x,y):
|x|=\epsilon\}$ is transverse to $C$ for all $\epsilon\le \epsilon_0$.
We define conical sets $V^{(j)}$ of the form
$$V^{(j)}:=\{(x,y):|y-a_1^{(j)}x|\le \eta |x|, |x|\le
\epsilon_0\}\subset \C^2\,,$$ where the equation of the line $L^{(j)}$
is $y=a_1^{(j)}x$ and $\eta>0$ is small enough that the cones are
disjoint except at $0$. If $\epsilon_0$ is small enough
$C^{(j)}\cap \{|x|\le\epsilon_0\}$ will lie completely in
$V^{(j)}$.  

There is then an $R>0$ such that for any $\epsilon\le \epsilon_0$ the
sets $V^{(j)}$ meet the boundary of the ``square ball''
$$B_\epsilon:=\{(x,y)\in \C^2: |x|\le \epsilon, |y|\le R\epsilon\}$$
only in the part $|x|=\epsilon$ of the boundary. We will use these
balls as a system of Milnor balls.

We now describe our carrousel decomposition for each $V^{(j)}$,  so we
will fix $j$ for the moment.

We first truncate the Puiseux series for each component of $C^{(j)}$
at a point where truncation does not affect the topology of
$C^{(j)}$. Then for each pair $\kappa=(f, p_k)$ consisting of a
Puiseux polynomial $f=\sum_{i=1}^{k-1}a^{(j)}_ix^{p^{(j)}_i}$ and an exponent
$p^{(j)}_k$ for which there is a Puiseux series
$y=\sum_{i=1}^{k}a^{(j)}_ix^{p^{(j)}_i}+\dots$ describing some component of
$C^{(j)}$, we consider all components of $C^{(j)}$ which fit this
data. If $a^{(j)}_{k1},\dots,a^{(j)}_{k{m_\kappa}}$ are the coefficients of
$x^{p^{(j)}_k}$ which occur in these Puiseux polynomials we define
\begin{align*}
  B_\kappa:=\Bigl\{(x,y):~&\alpha_\kappa|x^{p^{(j)}_k}|\le
  |y-\sum_{i=1}^{k-1}a^{(j)}_ix^{p^{(j)}_i}|\le
  \beta_\kappa|x^{p^{(j)}_k}|\\
  & |y-(\sum_{i=1}^{k-1}a^{(j)}_ix^{p^{(j)}_i}+a^{(j)}_{kj}x^{p^{(j)}_k})|\ge
  \gamma_\kappa|x^{p^{(j)}_k}|\text{ for }j=1,\dots,{m_\kappa}\Bigr\}\,.
\end{align*}
Here $\alpha_\kappa,\beta_\kappa,\gamma_\kappa$ are chosen so that
$\alpha_\kappa<|a^{(j)}_{k\nu}|-\gamma_\kappa<
|a^{(j)}_{k\nu}|+\gamma_\kappa<\beta_\kappa$ for each
$\nu=1,\dots,{m_\kappa}$.  If $\epsilon$ is small enough, the sets $
B_\kappa$ will be disjoint for different $\kappa$.

The intersection $B_{\kappa}\cap \{x=t\}$ is a finite collection of
disks with smaller disks removed. We call
$B_{\kappa}$ a  \emph{$B$-piece}.
The closure of the complement in $V^{(j)}$ of the union of the
$B_\kappa$'s is a union of pieces, each of which has link either a
solid torus or a ``toral annulus'' (annulus $\times~ \Bbb S^1$). We
call the latter \emph{annular pieces} or \emph{$A$-pieces} and the
ones with solid torus link \emph{$D$-pieces} (a $B$-piece
corresponding to an inessential exponent has the same topology as an
$A$-piece, but we do not call it annular).

This is our carrousel decomposition of $V=V^{(j)}$.  We call
$\overline{B_\epsilon \setminus \bigcup V^{(j)}}$ a $B(1)$ piece (even
though it may have $A$- or $D$-topology). It is metrically conical,
and together with the carrousel decompositions of the $V^{(j)}$'s we
get a carrousel decomposition of the whole of $B_\epsilon$.
     
\begin{proof}[Proof of \eqref{it3} $\Rightarrow$ \eqref{it4} of
  Theorem \ref{th:main}]
  Let $(C_1,0) \subset (\C^2,0)$ and $(C_2,0) \subset(\C^2,0)$ have
  the same embedded topological type. Consider two carrousel
  decompositions of $ (\C^2,0)$: one with respect to $C_1$ and the
  other with respect to $C_2$, constructed as above. The proof
  consists of constructing a bilipschitz map of germs $h \colon
  (\C^2,0) \to (\C^2,0)$ which sends the carrousel decomposition for
  $C_1$ to the one for $C_2$ (being careful to include matching pieces
  for inessential exponents which occur in just one of $C_1$ and
  $C_2$).  We first construct it to respect the carrousels, but
  not necessarily map $C_1$ to $C_2$. Once this is done, we 
  adjust it so that $C_1$ is mapped to $C_2$.

  Let $L^{(j)}_1$ and $L^{(j)}_2$, $j=1,\dots,m$, be the tangent lines
  to $C_1$ and $C_2$ and $C^{(j)}_1$ resp.\ $C^{(j)}_2$ the union of
  components of $C_1$ resp.\ $C_2$ which are tangent to $L^{(j)}_1$
  resp.\ $L^{(j)}_2$. We may assume we have numbered them so
  $C^{(j)}_1$ and $C^{(j)}_2$ have matching embedded topology. Let
  $V^{(j)}_1$ and $V^{(j)}_2$, $j=1,\dots,m$, be the conical sets
  around the tangent lines as defined earlier.

  The $B(1)$ pieces of the carrousel decompositions for $C_1$ and
  $C_2$ are metrically conical with the same topology, so there is a
  conical bilipschitz diffeomorphism between them. We can arrange that
  it is a translation on each $x=t$ section of each $\partial
  V_1^{(j)}$.  We will extend it over the cones $V_1^{(j)}$
  and $V_2^{(j)}$ using the carrousels.

  Consider the Puiseux series
  $y=\sum_{i=1}^{k}a^{(j)}_ix^{p^{(j)}_i}+\dots$ describing some
  component of $C_1^{(j)}$ and the Puiseux series
  $y=\sum_{i=1}^{k}b^{(j)}_ix^{p^{(j)}_i}+\dots$ describing the
  corresponding component of $C_2^{(j)}$.  If a term with inessential
  exponent appears in one of the series, we include it also in the
  other, even if its coefficient there is zero. This way, when we
  construct the carrousel as above we have corresponding
  $B$-pieces for the two carrousels. Moreover, we can choose the
  constants $\alpha_\kappa,\beta_\kappa,\gamma_\kappa$ used to 
  construct these corresponding $B$-pieces to be the same for
  both. The $\{x=t\}$ sections of a pair of corresponding $A$-pieces
  will then be congruent, so we can map the one $A$-piece to the other
  by preserving $x$ coordinate and using translation on each $x=t$
  section. The same holds for $D$-pieces. It then remains to extend to
  the $B$-pieces. 

  A $B$-piece $B_{\kappa_1}$ in the decomposition for $C_1$ is
  determined by some $\kappa_1=(f_1, p_k)$ with
  $f_1=\sum_{i=1}^{k-1}a_ix^{p_i}$, and is foliated by curves of
  the form $y=f_1+\xi x^{p_k}$ for varying $\xi$. The corresponding
  piece $B_{\kappa_2}$ for $C_2$ is similarly determined by some
  $\kappa_2=(f_2, p_k)$ with $f_2=\sum_{i=1}^{k-1}b_ix^{p_i}$ and is
  foliated by curves $y=f_2+\xi x^{p_k}$. The $x=\epsilon_0$ section
  of $B_{\kappa_1}$ has a free cyclic group action generated by the
  first return map of the foliation, and the same is true for
  $B_{\kappa_2}$. We choose a smooth map
  $(B_{\kappa_1}\cap\{x=\epsilon_0\})\to(B_{\kappa_2}\cap\{x=\epsilon_0\})$
  which is equivariant for this action and on the boundary matches the
  maps, coming from $A$- and $D$-pieces, already chosen. This map
  extends to the whole of $B_{\kappa_1}$ by requiring it to preserve the
  foliation and $x$-coordinate. 

  It is not hard to verify that the resulting map of germs
  $\phi\colon(\C^2,0)\to (\C^2,0)$ is bilipschitz. However, it maps
  $C_1$ not to $C_2$, but to a small deformation of it, since we
  constructed the carrousels by first truncating our Puiseux series
  beyond any terms which contributed to the topology. But again it is
  not hard to see that, by a small change of the constructed map
  inside the $D$-pieces which intersect $C_1$, one can change $\phi$
  so it maps $C_1$ to $C_2$ while changing the bilipschitz coefficient
  by an amount which approaches zero as one approaches the origin.
 \end{proof}

\end{document}